\documentclass{gen-j-l}

\usepackage{tikz}
\usepackage{amscd}
\usepackage{textcomp}
\usepackage{gensymb}
\usepackage{csquotes}
\usepackage{amssymb}
\usepackage{calc}

\numberwithin{figure}{section}

\copyrightinfo{2018}{Fabian Parsch}

\newtheorem{theorem}{Theorem}[section]
\newtheorem{lemma}[theorem]{Lemma}

\newcommand{\degrees}{\degree}
  
\tikzset{
	c/.style={every coordinate/.try}
}
  
\begin{document}

\title{An example for a nontrivial irreducible geodesic net in the plane}

\author{Fabian Parsch}
\address{Fabian Parsch, Department of Mathematics, University of Toronto, 40 St. George Street, Toronto, ON M5S 2E4, Canada}
\email{fparsch@math.toronto.edu}

\begin{abstract}
	We construct a geodesic net in the plane with four unbalanced (boundary) vertices that has $16$ balanced vertices and does not contain proper geodesic subnets. This is the first example of an irreducible geodesic net in the Euclidean plane with $4$ boundary vertices that is not a tree.
\end{abstract}

\maketitle

\section{Introduction}

Geodesic nets in Riemannian manifolds are critical points (but not necessarily local minima) of the length functional on spaces of embedded graphs or multigraphs, where some vertices must be mapped to prescribed vertices that are called {\it boundary} or {\it unbalanced}.  As such they can be regarded as generalizations of geodesics (that arise in the case of two boundary points) as well as $1$-dimensional analogs of minimal surfaces spanning a given contour. Surprisingly, very little is known about the classification of geodesic nets even in the case when the ambient Riemannian manifold is just the Euclidean plane. In this case the geodesic nets are simply embedded graphs so that (i) each edge is a straight line segment; (ii) Some of the edges are mapped to prescribed points (\enquote{unbalanced vertices}) and (iii) At each of the remaining vertices $v$ (\enquote{balanced vertices}) the following balancing condition holds: The sum of all unit vectors directed along all edges from $v$ to the opposite end of the edge is equal to the zero vector. (This characterization holds also in the more general case of ambient Riemannian manifolds with the only distinction that edges are supposed to be geodesics.)

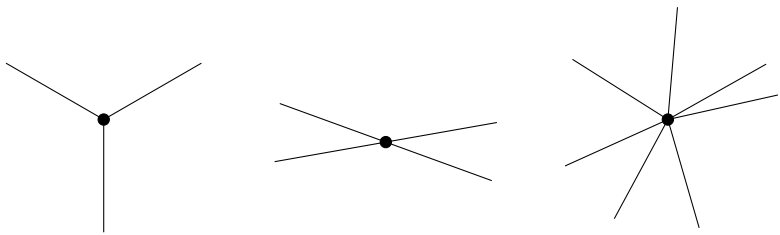
\begin{figure}
	\centering
	\begin{tikzpicture}[scale=1.5]
		\draw (-1,0) -- +(30:1);
		\draw (-1,0) -- +(150:1);
		\draw (-1,0) -- +(270:1);
		\draw [fill=black] (-1,0) circle [radius=0.05];
		
		\draw (1.5,-0.2) -- +(10:1);
		\draw (1.5,-0.2) -- +(-20:1);
		\draw (1.5,-0.2) -- +(160:1);
		\draw (1.5,-0.2) -- +(190:1);
		\draw [fill=black] (1.5,-0.2) circle [radius=0.05];
		
		\draw (4,0) -- (4.085159600262456,0.9963673230707331);
		\draw (4,0) -- (3.154511096969029,0.5339929913879817);
		\draw (4,0) -- (3.088735454976287,-0.4118214770780241);
		\draw (4,0) -- (4.975942150302623,0.2180296293229263);
		\draw (4,0) -- (3.5248510852651163,-0.8799053976571926);
		\draw (4,0) -- (4.2773364928491535,-0.9607728502274256);
		\draw (4,0) -- (4.870934923071531,0.49139837176611095);
		\draw [fill=black] (4,0) circle [radius=0.05];
	\end{tikzpicture}
	\caption{Examples for balanced vertices of degree $3$, $4$ and $7$. }
	\label{fig:balancedvertices}
\end{figure}

Note that in the present paper we are not allowing integer multiplicities of edges (or, equivalently, we assume that each edge has multiplicity $1$). As one can add and remove vertices of degree $2$ inside each edge at will, the role of such vertices in the classification is clear, and we will be assuming that {\bf all balanced vertices have degree greater than or equal to $3$.} Also, as one can add or remove straight line edges between unbalanced vertices, we will consider only geodesic nets without such edges. Finally, we are going to consider only connected geodesic nets.

As geodesics nets in the plane are obviously contained in the convex hull of the unbalanced point, the classification really starts from the case of $3$ balanced vertices. The first non-trivial example arises when three unbalanced points $A_1$, $A_2$, $A_3$ form a triangle with all angles less than $120\degrees$. In this case there exists the unique point $F$ called the Fermat point of the triangle $A_1A_2A_3$, such that all angles $A_iFA_j$, $i\not= j$, are equal to $120\degrees$. This condition implies that $F$ is the balanced vertex in the geodesic net formed by the three straight line segments connecting $F$ with $A_i$ for $i=1,2,3$.

In \cite{Parsch:2018aa} we prove the theorem asserting that this is the only possible example of a geodesic net with three unbalanced vertices in the Euclidean plane (and, more generally, any Riemannian manifold endowed with a non-positive curved Riemannian metric). This theorem is surprisingly difficult to prove. It can be stated as follows:

\begin{theorem}[Geodesic nets with 3 unbalanced vertices, \cite{Parsch:2018aa}]\label{thm:three}
	Each geodesic net with 3 unbalanced vertices (of arbitrary degree) on the plane endowed with a Riemannian metric of non-positive curvature has exactly one balanced vertex.
\end{theorem}

What happens in the next case, when the number of unbalanced vertices is equal to $4$? Consider first the geodesic net in figure \ref{fig:fournet} (that we first constructed in \cite{Parsch:2018aa}). At first sight, it seems rather complicated. However, the figure shows that it is essentially an \enquote{overlay} of several geodesic nets, each of them being a tree: four geodesic nets with $3$ unbalanced vertices and the Fermat point in the middle, and three well-known tree-shaped geodesic nets with $4$ unbalanced vertices and $1$ or $2$ balanced vertices (plus the balanced vertices that appear as points of the intersections of edges of these elementary geodesic nets). With this example in mind, it is of interest to define {\it irreducible} geodesic nets as geodesic nets without proper geodesic subnets. (Here $G_1$ is a proper gedesic subnet of $G$ if (i) The set of balanced (resp. unbalanced) vertices of $G_1$ is the set of balanced (resp. unbalanced) vertices of $G$; (ii) The set of edges of $G_1$ is a non-empty subset of the set $E$ of edges of $G$ that does not coincide with $E$.) Then we immediately realize that the only known examples of irreducible geodesic nets with four unbalanced vertices are the trees with $1$ or $2$ balanced vertices that can be seen at the right side of figure 1.2.

This brings us to the questions that we are trying to answer in this paper:

{\bf Question.} Do there exist irreducible geodesic nets with $4$ unbalanced vertices in the Euclidean plane with at least $3$ balanced vertices? Can they contain cycles of balanced points?

The main result of the paper is that the answer for these questions is yes. 

{\bf Main Theorem.} There exists an irreducible geodesic net in the Euclidean plane that has $16$ balanced vertices and $4$ unbalanced vertices.

It is tempting to conjecture that our example is one of a series of similar examples with arbitrary large number of balanced vertices, but at the moment this is the only new example of an irreducible geodesic net with four unbalanced vertices that we were able to construct.

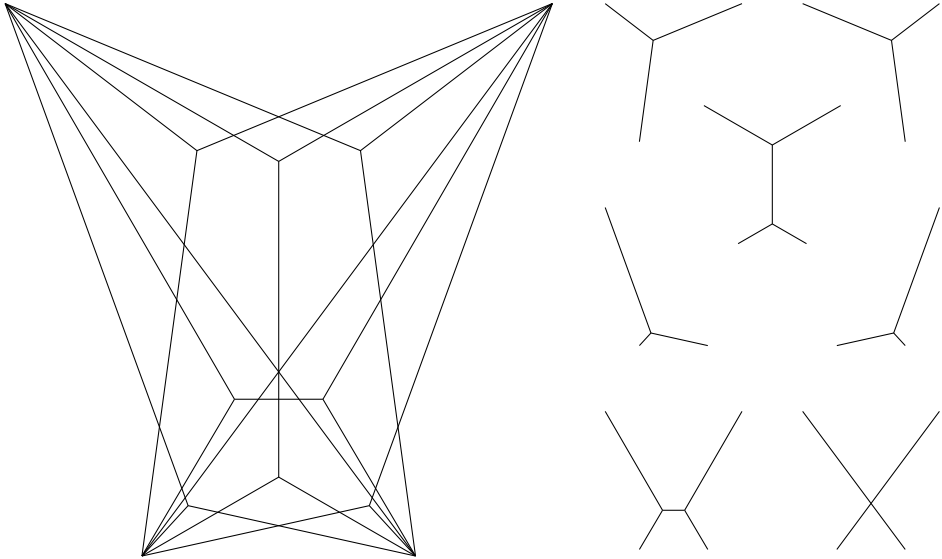
\begin{figure}[t]
	\centering
	\begin{tikzpicture}[scale=0.7]
	\clip (-5.5,1.5) rectangle (20,12);
	\node at (0,0) (Q) {};
	\node at (150:3) (X) {};
	\node at (125:3) (Y1) {};
	\node at (90:3) (Y2) {};
	\node at (55:3) (Y3) {};
	\node at (30:3) (Z) {};
	
	\node at (0,15) (P) {};
	\path (P) -- ++(210:6) node (A) {};
	\path (P) -- ++(255:6) node (B1) {};
	\path (P) -- ++(270:6) node (B2) {};
	\path (P) -- ++(285:6) node (B3) {};
	\path (P) -- ++(330:6) node (C) {};
	
	\node at (-.84,4.48) (L) {};
	\node at (.84,4.48) (N) {};

	\draw (A.center) -- (B1.center) -- (C.center);
	\draw (A.center) -- (B2.center) -- (C.center);
	\draw (A.center) -- (B3.center) -- (C.center);
	
	\draw (X.center) -- (Y1.center) -- (Z.center);
	\draw (X.center) -- (Y2.center) -- (Z.center);
	\draw (X.center) -- (Y3.center) -- (Z.center);
	
	\draw (B1.center) -- (X.center);
	\draw (B3.center) -- (Z.center);
	
	\draw (Y1.center) -- (A.center);
	\draw (Y3.center) -- (C.center);
	
	\draw (B2.center) -- (Y2.center);
	
	\draw (A.center) -- (L.center);
	\draw (X.center) -- (L.center);
	\draw (Z.center) -- (N.center);
	\draw (C.center) -- (N.center);
	\draw (L.center) -- (N.center);
	
	\draw (A.center) -- (Z.center);
	\draw (C.center) -- (X.center);
	\begin{scope}[scale=0.25, yshift=36cm, xshift=30cm]
	\node at (0,0) (Q) {};
	\node at (150:3) (X) {};
	\node at (125:3) (Y1) {};
	\node at (90:3) (Y2) {};
	\node at (55:3) (Y3) {};
	\node at (30:3) (Z) {};
	
	\node at (0,15) (P) {};
	\path (P) -- ++(210:6) node (A) {};
	\path (P) -- ++(255:6) node (B1) {};
	\path (P) -- ++(270:6) node (B2) {};
	\path (P) -- ++(285:6) node (B3) {};
	\path (P) -- ++(330:6) node (C) {};
	
	\node at (-.84,4.48) (L) {};
	\node at (.84,4.48) (N) {};

	\draw (A.center) -- (B1.center) -- (C.center);
	
	\draw (B1.center) -- (X.center);
	\end{scope}
	\begin{scope}[scale=0.25, yshift=36cm, xshift=45cm]
	\node at (0,0) (Q) {};
	\node at (150:3) (X) {};
	\node at (125:3) (Y1) {};
	\node at (90:3) (Y2) {};
	\node at (55:3) (Y3) {};
	\node at (30:3) (Z) {};
	
	\node at (0,15) (P) {};
	\path (P) -- ++(210:6) node (A) {};
	\path (P) -- ++(255:6) node (B1) {};
	\path (P) -- ++(270:6) node (B2) {};
	\path (P) -- ++(285:6) node (B3) {};
	\path (P) -- ++(330:6) node (C) {};
	
	\node at (-.84,4.48) (L) {};
	\node at (.84,4.48) (N) {};

	\draw (A.center) -- (B3.center) -- (C.center);
	
	\draw (B3.center) -- (Z.center);
	\end{scope}
	\begin{scope}[scale=0.25, yshift=28.25cm, xshift=37.5cm]
	\node at (0,0) (Q) {};
	\node at (150:3) (X) {};
	\node at (125:3) (Y1) {};
	\node at (90:3) (Y2) {};
	\node at (55:3) (Y3) {};
	\node at (30:3) (Z) {};
	
	\node at (0,15) (P) {};
	\path (P) -- ++(210:6) node (A) {};
	\path (P) -- ++(255:6) node (B1) {};
	\path (P) -- ++(270:6) node (B2) {};
	\path (P) -- ++(285:6) node (B3) {};
	\path (P) -- ++(330:6) node (C) {};
	
	\node at (-.84,4.48) (L) {};
	\node at (.84,4.48) (N) {};

	\draw (A.center) -- (B2.center) -- (C.center);
	
	\draw (X.center) -- (Y2.center) -- (Z.center);
	
	\draw (B2.center) -- (Y2.center);
	\end{scope}
	
	\begin{scope}[scale=0.25, yshift=20.5cm, xshift=30cm]
	\node at (0,0) (Q) {};
	\node at (150:3) (X) {};
	\node at (125:3) (Y1) {};
	\node at (90:3) (Y2) {};
	\node at (55:3) (Y3) {};
	\node at (30:3) (Z) {};
	
	\node at (0,15) (P) {};
	\path (P) -- ++(210:6) node (A) {};
	\path (P) -- ++(255:6) node (B1) {};
	\path (P) -- ++(270:6) node (B2) {};
	\path (P) -- ++(285:6) node (B3) {};
	\path (P) -- ++(330:6) node (C) {};
	
	\node at (-.84,4.48) (L) {};
	\node at (.84,4.48) (N) {};

	\draw (X.center) -- (Y1.center) -- (Z.center);
	
	\draw (Y1.center) -- (A.center);
	\end{scope}
	\begin{scope}[scale=0.25, yshift=20.5cm, xshift=45cm]
	\node at (0,0) (Q) {};
	\node at (150:3) (X) {};
	\node at (125:3) (Y1) {};
	\node at (90:3) (Y2) {};
	\node at (55:3) (Y3) {};
	\node at (30:3) (Z) {};
	
	\node at (0,15) (P) {};
	\path (P) -- ++(210:6) node (A) {};
	\path (P) -- ++(255:6) node (B1) {};
	\path (P) -- ++(270:6) node (B2) {};
	\path (P) -- ++(285:6) node (B3) {};
	\path (P) -- ++(330:6) node (C) {};
	
	\node at (-.84,4.48) (L) {};
	\node at (.84,4.48) (N) {};

	\draw (X.center) -- (Y3.center) -- (Z.center);
	
	\draw (Y3.center) -- (C.center);
	\end{scope}
	
	\begin{scope}[scale=0.25, yshift=5cm, xshift=30cm]
	\node at (0,0) (Q) {};
	\node at (150:3) (X) {};
	\node at (125:3) (Y1) {};
	\node at (90:3) (Y2) {};
	\node at (55:3) (Y3) {};
	\node at (30:3) (Z) {};
	
	\node at (0,15) (P) {};
	\path (P) -- ++(210:6) node (A) {};
	\path (P) -- ++(255:6) node (B1) {};
	\path (P) -- ++(270:6) node (B2) {};
	\path (P) -- ++(285:6) node (B3) {};
	\path (P) -- ++(330:6) node (C) {};
	
	\node at (-.84,4.48) (L) {};
	\node at (.84,4.48) (N) {};

	\draw (A.center) -- (L.center);
	\draw (X.center) -- (L.center);
	\draw (Z.center) -- (N.center);
	\draw (C.center) -- (N.center);
	\draw (L.center) -- (N.center);
	\end{scope}
	\begin{scope}[scale=0.25, yshift=5cm, xshift=45cm]
	\node at (0,0) (Q) {};
	\node at (150:3) (X) {};
	\node at (125:3) (Y1) {};
	\node at (90:3) (Y2) {};
	\node at (55:3) (Y3) {};
	\node at (30:3) (Z) {};
	
	\node at (0,15) (P) {};
	\path (P) -- ++(210:6) node (A) {};
	\path (P) -- ++(255:6) node (B1) {};
	\path (P) -- ++(270:6) node (B2) {};
	\path (P) -- ++(285:6) node (B3) {};
	\path (P) -- ++(330:6) node (C) {};
	
	\node at (-.84,4.48) (L) {};
	\node at (.84,4.48) (N) {};
	
	\draw (A.center) -- (Z.center);
	\draw (C.center) -- (X.center);	
	\end{scope}

\end{tikzpicture}
	\caption{An example of a geodesic net in the plane with four unbalanced vertices. While this net might look complicated at first, it is just a union of seven trees as depicted. In other words, in is an \enquote{overlay} of more elementary nets and therefore not \emph{irreducible}}
	\label{fig:fournet}
\end{figure}

\section{Construction of the example}

Before we start, it is worth noting that the geodesic net $G$ we are constructing here will be symmetric under a rotation by $90\degrees$.

In this section, the notation $ABC$ means the angle at $B$ from $A$ to $C$. Furthermore $\mathcal{O}=(0,0)$ will denote the origin. The end result of the construction is given in figure \ref{fig:fullnet}.

\subsection{The octagon of $a_i$ and $b_i$}

Fix four vertices at $a_1=(1,0)$, $a_2=(0,1)$, $a_3=(-1,0)$ and $a_4=(0,-1)$. Now add another four vertices $b_1$, $b_2$, $b_3$, $b_4$ so that we arrive at an octagon $a_1b_2a_2b_2a_3b_3a_4b_4$ where the interior angle at each $a_i$ is $150\degrees$ and the interior angle at each $b_i$ is $120\degrees$.

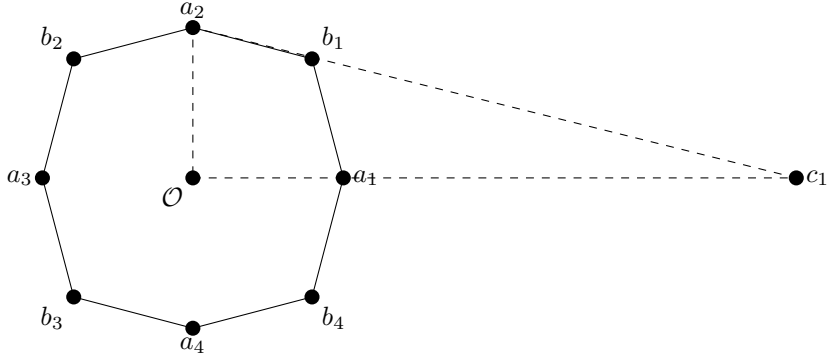
\begin{figure}
	\centering
	\begin{tikzpicture}[scale=2]
		\fill (0,0) circle (0.05) node [below left=0.1] {$\mathcal{O}$};
	
		\draw (0:1) coordinate (a1) -- (45:1.12) coordinate (b1) -- (90:1) coordinate (a2) -- (135:1.12) coordinate (b2) -- (180:1) coordinate (a3) -- (225:1.12) coordinate (b3) -- (270:1) coordinate (a4) -- (315:1.12) coordinate (b4) -- cycle;
		
		\fill (a1) circle (0.05) node [right=0.2] {$a_1$};
		\fill (a2) circle (0.05) node [above=0.2] {$a_2$};
		\fill (a3) circle (0.05) node [left=0.2] {$a_3$};
		\fill (a4) circle (0.05) node [below=0.2] {$a_4$};
		
		\fill (b1) circle (0.05) node [above right=0.2] {$b_1$};
		\fill (b2) circle (0.05) node [above left=0.2] {$b_2$};
		\fill (b3) circle (0.05) node [below left=0.2] {$b_3$};
		\fill (b4) circle (0.05) node [below right=0.2] {$b_4$};
		
		\fill (0:{tan(76)}) coordinate (c1) circle (0.05) node [right=0.2] {$c_1$};
		
		\draw [dashed] (0,0) -- (c1) -- (a2) -- cycle;
	\end{tikzpicture}
	\caption{The octagon in the first step with alternating interior angles of $150\degrees$ and $120\degrees$. The dashed triangle is used to define the position of $c_1$. Note that the side $a_2c_1$ does \emph{not} go through $b_1$. In fact, the angle $\mathcal{O}a_2b_1$ is $75\degrees$ whereas the angle $\mathcal{O}a_2c_1$ is just over $76\degrees$.}
	\label{fig:octagon}	
\end{figure}

\subsection{The four unbalanced vertices $c_i$}
There is a uniquely defined triangle as follows, see figure \ref{fig:octagon}:
\begin{itemize}
	\item The left side is $a_2\mathcal{O}$.
	\item The angle at $\mathcal{O}$ is $90\degrees$.
	\item The angle at $a_2$ is $\arccos\left(\frac12-\cos 75\degrees\right)\approx 76.04\degrees$.
\end{itemize}
The resulting third vertex of this triangle is denoted by $c_1$. By rotation around $\mathcal{O}$ we get vertices $c_2$, $c_3$ and $c_4$, see figure \ref{fig:fullnet}.

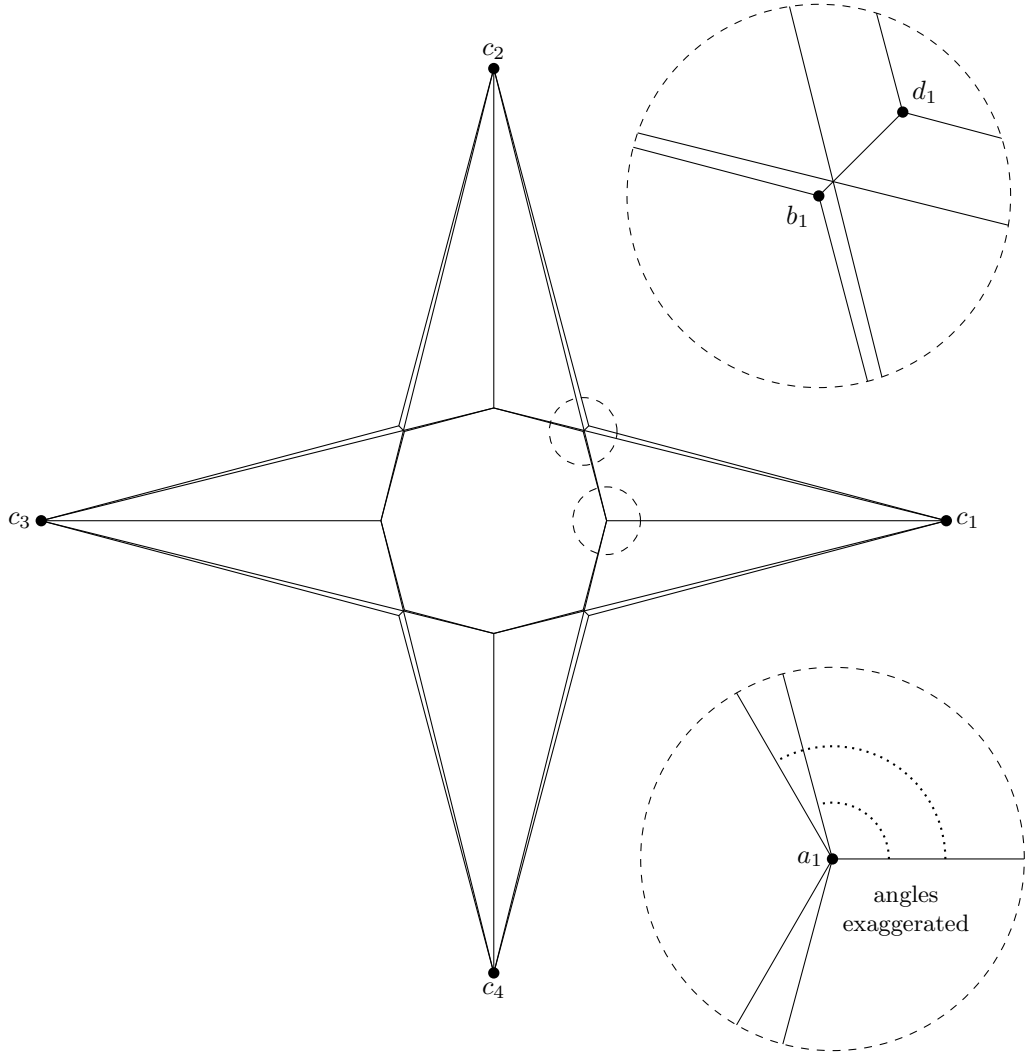
\begin{figure}
	\centering
	\begin{tikzpicture}[scale=1.5]	
		\draw (0:1) coordinate (a1) -- (45:1.12) coordinate (b1) -- (90:1) coordinate (a2) -- (135:1.12) coordinate (b2) -- (180:1) coordinate (a3) -- (225:1.12) coordinate (b3) -- (270:1) coordinate (a4) -- (315:1.12) coordinate (b4) -- cycle;
		
		\fill (0:{tan(76)}) coordinate (c1) circle (0.05) node [right=0.2] {$c_1$};
		\fill (90:{tan(76)}) coordinate (c2) circle (0.05) node [above=0.2] {$c_2$};
		\fill (180:{tan(76)}) coordinate (c3) circle (0.05) node [left=0.2] {$c_3$};
		\fill (270:{tan(76)}) coordinate (c4) circle (0.05) node [below=0.2] {$c_4$};
		
		\draw (a1) -- (c2) -- (a3) -- (c4) -- (a1);
		\draw (a2) -- (c3) -- (a4) -- (c1) -- (a2);
		
		\draw (a1) -- (c1);
		\draw (a2) -- (c2);
		\draw (a3) -- (c3);
		\draw (a4) -- (c4);
		
		\draw (b1) -- ++ (45:0.07) coordinate (d1);
		\draw (c1) -- (d1) -- (c2);
		
		\draw (b2) -- ++ (135:0.07) coordinate (d2);
		\draw (c2) -- (d2) -- (c3);
		
		\draw (b3) -- ++ (225:0.07) coordinate (d3);
		\draw (c3) -- (d3) -- (c4);
		
		\draw (b4) -- ++ (315:0.07) coordinate (d4);
		\draw (c4) -- (d4) -- (c1);
		
		\draw [dashed] (b1) circle (0.3);
		\draw [dashed] (a1) circle (0.3);
		
		\begin{scope}[every coordinate/.style={shift={(-9,-9)},scale=15}]
			\begin{scope}
			\clip ([c]b1) circle (1.7);
			\draw ([c]a1) -- ([c]b1) -- ([c]a2);
			\draw ([c]a1) -- ([c]c2);
			\draw ([c]a2) -- ([c]c1);
			\draw ([c]b1) -- ([c]d1);
			\draw ([c]c1) -- ([c]d1) -- ([c]c2);
			\fill ([c]b1) circle (0.05) node [below left] {$b_1$};
			\fill ([c]d1) circle (0.05) node [above right] {$d_1$};
			\end{scope}
			\draw [dashed] ([c]b1) circle (1.7);
		\end{scope}
		
		\begin{scope}[every coordinate/.style={shift={(-12,-3)},scale=15}]
			\begin{scope}
			\clip ([c]a1) circle (1.7);
			\draw ([c]a1) -- ++ (120:3);
			\draw ([c]a1) -- ++ (-120:3);
			\draw ([c]a1) -- ++ (105:3);
			\draw ([c]a1) -- ++ (-105:3);
			\draw ([c]a1) -- ([c]c1);
			\fill ([c]a1) circle (0.05) node [left] {$a_1$};
			\path ([c]a1) -- ++(325:0.8) node [text width=2cm,align=center] {\small angles exaggerated\par};
			\path ([c]a1) -- ++(0:0.5) coordinate (arcstart1);
			\path ([c]a1) -- ++(0:1) coordinate (arcstart2);
			\draw [dotted,thick] (arcstart1) arc (0:105:0.5);
			\draw [dotted,thick] (arcstart2) arc (0:120:1);
			\end{scope}
			\draw [dashed] ([c]a1) circle (1.7);
		\end{scope}

	\end{tikzpicture}
	\caption{The complete geodesic net with the unbalanced vertices $c_1$, $c_2$, $c_3$, $c_4$. The zoom-ins are provided for overview of the balanced vertices. The two dotted angles at $a_1$ are $180\degrees-75\degrees=105\degrees$ and $180\degrees-\arccos\left(\frac12-\cos 75\degrees\right)\approx 103.96\degrees$ respectively. In the zoom-in, they are exaggerated for clarity.}
	\label{fig:fullnet}	
\end{figure}

\subsection{The fermat points $d_i$}

We define $d_1$ as follows (again, $d_2$, $d_3$ and $d_4$ will be defined by rotational symmetry): It is the Fermat point of the triangle $b_1c_1c_2$. Recall that the Fermat point is the unique point $x$ in a triangle such that the angle at $x$ between any two corners of the triangle is $120\degrees$. It exists as long as all interior angles of the triangle are less than $120\degrees$. So we are left to show:
\begin{lemma}
	All three interior angles of the triangle $b_1c_1c_2$ are less than $120\degrees$.
\end{lemma}

\begin{proof}	
	It is obvious that the interior angles at $c_1$ and $c_2$ are less than $120\degrees$ (in fact they are both significantly smaller than $90\degrees$). So it remains to show that the angle at $b_2$ is less than $120\degrees$. This can be shown as follows:
	\begin{itemize}
		\item $c_1Ob_1=45\degrees$ follows from the symmetries of the octagon.
		\item The segment $a_2c_1$ is above the segment $b_1c_1$ (see figure \ref{fig:octagon}). It follows that
		\begin{align*}
			 b_1c_1O< a_1c_1O\approx 180\degrees-90\degrees-76.04\degrees=13.96
		\end{align*}
		\item We now have estimates for two of the angles of the triangle with corners $\mathcal{O}$, $c_1$, $b_1$ and get
		\begin{align*}
			 Ob_1c_1&=180\degrees- c_1Ob_1- b_1c_1O>180\degrees- c_1Ob_1- a_1c_1\mathcal{O}\\
			 &\approx 180\degrees-45\degrees-13.96\degrees=121.04\degrees
		\end{align*}
		\item We can us the symmetry of the picture and conclude:
		\begin{align*}
			 c_1b_1c_2&=360\degrees -  Ob_1c_1 -  c_2b_1O \\
			 &= 360\degrees - 2\cdot  Ob_1c_1\approx 117.92<120\degrees
		\end{align*}
	\end{itemize}	
\end{proof}

\subsection{The edges}

Finally we add the edges of the geodesic net. We recommend referring again to figure \ref{fig:fullnet} for better understanding.

The following definitions have to be read circular, e.g. $a_5$ is the same as $a_1$. With that in mind, the edges of the geodesic net for $i=1,2,3,4$ are:
\begin{itemize}
	\item The edges of the octagon, given by $a_ib_i$ and $a_{i+1}b_i$.
	\item The radial edges given by $a_ic_i$.
	\item The edges given by $a_ic_{i+1}$ and $a_{i+1}c_i$.
	\item Finally, the edges for each Fermat point, given by $d_ib_i$, $d_ic_i$ and $d_ic_{i+1}$.
\end{itemize}

\subsection{The additional vertices $x_i$}

Note that we are getting four additional vertices of degree $6$ which are situated on the edge connecting $b_1$ and $d_1$ (again, see figure \ref{fig:fullnet}). We call these vertices $x_1$, $x_2$, $x_3$ and $x_4$.

\subsection{Checking the edges and the balance}

\begin{lemma}
	The net constructed above is a valid geodesic net with four unbalanced vertices, more specifically:
	\begin{itemize}
		\item Different edges can intersect only at their common endpoints (in other words: there never is an \enquote{overlay} of edges, so all edges have weight one).
		\item Each $a_i$, $b_i$, $d_i$, $x_i$ is balanced.
	\end{itemize}
\end{lemma}

\begin{proof}
	Note again that the picture is rotationally symmetric by design. So we can concentrate on the corner $c_1Oc_2$ (the upper right corner).
	
	As long as we prove that none of the edges are parallel, the result follows. We will go through the edges as defined above, again adding them step by step. 
	\begin{itemize}
		\item It is apparent that none of $a_1b_1$, $a_2b_1$ and $a_1c_1$ are.
		\item Adding $a_1c_2$, note that the angle between $a_1b_2$ and $a_1c_2$ is approximately $1.04\degrees$, so these two edges are not parallel. It is apparent that $a_1c_2$ is never parallel to any other edge. By symmetry, adding $a_2c_1$ doesn't create issues either.
		\item The edge $d_1b_1$ is radial at the angle $45\degrees$. No other edge is. This finally brings us to the only two interesting edges: adding $d_1c_1$ and $d_1c_2$. We will consider the former. Symmetry will then deal with the latter. The only problem could arise if $d_1c_1$ coincides with the previously added $a_2c_1$ (which would also imply that $d_1=x_1$). Elementary calculations involving the angle sum in triangles, however, show that $d_1c_1O=15\degrees$ whereas $a_2c_1O\approx 13.96\degrees$. It follows that the two edges in question are not parallel.
	\end{itemize}
	
	We finish with showing that all vertices except the $c_i$ are balanced. By symmetry, it is again enough to consider $i=1$:
	\begin{itemize}
		\item Each of the $a_1$ is a degree $5$ balanced vertex. Putting the origin of the coordinate system at $a_1$, the sum of the unit vectors parallel to the five edges can be written as follows (refer to the zoom-in in figure \ref{fig:fullnet}):
		\begin{align*}
			\langle 1,0\rangle &+\langle \cos(180\degrees-\arccos(1/2-\cos(75\degrees))),\sin(180\degrees-\arccos(1/2-\cos(75\degrees)))\rangle\\
			&+ \langle \cos(180\degrees-75\degrees),\sin(180\degrees-75\degrees)\rangle\\
			&+\langle \cos(180\degrees+\arccos(1/2-\cos(75\degrees))),\sin(180\degrees+\arccos(1/2-\cos(75\degrees)))\rangle\\
			&+ \langle \cos(180\degrees+75\degrees),\sin(180\degrees+75\degrees)\rangle=\langle 0,0\rangle
		\end{align*}
		In fact, the very reason for choosing the \enquote{odd angle} $\arccos(1/2-\cos(75\degrees))$ early in the construction was to ensure that the $a_i$ are balanced.
		\item $b_1$ is a degree three balanced vertex. This follows from the fact that two of the incident edges belong to the octagon, so the angle between them at $b_1$ is $120\degree$. The third edge at $b_1$ is the bisector of the larger angle between the other two edges by symmetry. The balancedness of $b_1$ follows.
		\item $d_1$ is balanced by the definition of a Fermat point.
		\item Finally $x_1$ is just the point of intersection of several straight edges and is trivially balanced.
	\end{itemize}
\end{proof}

\section{Proof that $G$ is irreducible}

While it is obvious that $G$ is not a tree, we need to show that:

\begin{lemma}
	$G$ is irreducible.
\end{lemma}

\begin{proof}
	We are going to give a proof by contradiction. Assume that $G_1$ is a proper geodesic subnet of $G$. First, assume that the set of (balanced) vertices of $G_1$ does not contain any $a_i$. Then $G_1$ does not contain any edges incident to $a_i$. Now it is easy to see that $G_1$ does not contain any vertices $b_i$ as well as edges incident to $b_i$. From here it is easy to see that $G_1$ is empty - a contradiction.
	
	So, we can assume without any loss of generality that $a_1$ is a (balanced) vertex of $G_1$.

	A simple check of the $2^5$ subsets of edges incident to $a_1$ (there are many symmetric cases) show that the only way that $a_1$ can be balanced is if \emph{all} incident edges are used. It follows that $G_1$ includes all vertices adjacent to $a_1$.
	
	We therefore know that $a_1,b_1,b_4,c_1,c_2,c_4$ are in the vertex set of $G_1$.
	
	Consider $b_1$ which is a degree three vertex. Obviously one can't take a proper subset of the set of incident edges to balance $b_1$ (and the same will be true for all degree $3$ vertices ). It follows that $G_1$ includes all vertices adjacent to $b_1$.
	
	We therefore know that $a_1,a_2,b_1,b_4,c_1,c_2,c_4,d_1$ are in the vertex set of $G_1$.
	
	Now that $a_2$ is in the net, we can reuse the argument based on $a_1$ above, adding $b_2$ and $c_3$ to the picture. Again, reuse previous arguments for $b_2$ and it follows that $d_2$ and $a_3$ are part of $G_1$. It should now be apparent how to conclude that $c_4$, $b_3$, $d_3$, $a_4$ and $d_4$ are in $G_1$.
	
	So $G_1$ includes all balanced vertices of $G$, except possibly the $x_i$. However, as previously argued, since all the $b_i$ and $d_i$ are of degree $3$, all their incident edges are in $G_1$. Also since the $a_i$ can only be balanced with all incident edges included, all edges of $G$ are in $G_1$. Since the $x_i$ are just points of intersection of edges, they are also in $G$.
	
	We can conclude that $G_1=G$. So, $G_1$ is not proper, and we obtain the desired contradiction. Hence, $G$ is irreducible.
\end{proof}

\section{Acknowledgements}

The author would like to thank his PhD advisor Alexander Nabutovsky for insightful discussions and for suggestions to improve the exposition of this paper. This research was partially supported by an NSERC Vanier Scholarship.

\bibliographystyle{amsalpha}

\bibliography{../../bib/main}

\end{document}